\documentclass{article}
\usepackage{amsmath,amsfonts,amssymb,amsthm}
\usepackage{amsmath,amsthm,amssymb,amsfonts,graphicx,amsthm,nicefrac,mathtools}
\usepackage{enumitem}
\usepackage{pstricks-add}
\usepackage[round]{natbib}
\usepackage{graphicx}
\usepackage{paralist}
 
\allowdisplaybreaks

\newcommand{\ca}{{\cal A}}

\newcommand{\cp}{{\cal P}}

\newcommand{\cs}{{\cal S}}
\newcommand{\ct}{{\cal T}}

\newcommand{\bsa}{\boldsymbol{a}}
\newcommand{\bsb}{\boldsymbol{b}}
\newcommand{\bsc}{\boldsymbol{c}}
\newcommand{\bsm}{\boldsymbol{m}}

\newcommand{\bsx}{\boldsymbol{x}}
\newcommand{\bsy}{\boldsymbol{y}}
\newcommand{\bsz}{\boldsymbol{z}}

\newcommand{\bsU}{\boldsymbol{U}}

\newcommand{\ints}{\mathbb{Z}}

\newcommand{\real}{\mathbb{R}}

\newcommand{\tran}{\mathsf{T}} 

\newcommand{\phm}{\phantom{-}}

\newcommand{\rd}{\mathrm{\, d}}

\newcommand{\one}{{\mathbf{1}}}

\newcommand{\vol}{{\mathbf{vol}}}

\renewcommand{\emptyset}{\varnothing}
\renewcommand{\ge}{\geqslant}
\renewcommand{\le}{\leqslant}

\newcommand{\dustd}{\mathbf{U}} 

\newcommand{\e}{{\mathbb{E}}} 

\newcommand{\hk}{{\mathrm{HK}}}

  \newcommand{\BIT}{\begin{itemize}}
\newcommand{\EIT}{\end{itemize}}
\newcommand{\BNUM}{\begin{enumerate}}
\newcommand{\ENUM}{\end{enumerate}}

\theoremstyle{definition}
\newtheorem{definition}{Definition}

\theoremstyle{plain}
\newtheorem{theorem}{Theorem}
\newtheorem{lemma}{Lemma}
\theoremstyle{definition}

\newcommand{\calt}{\mathcal{T}}
\newcommand{\pr}{{\mathrm{P}}}
\newcommand{\pc}{{\mathrm{PC}}}

\newcommand{\etri}{{\Delta_E}}

\newcommand{\ptA}{A}
\newcommand{\ptB}{B}
\newcommand{\ptC}{C}

\date{March 2014}
\author{Kinjal Basu\\Stanford University \and Art B. Owen\\Stanford University}
\title{Low discrepancy constructions in the triangle}
\begin{document}
\maketitle
\begin{abstract}
Most quasi-Monte Carlo research focuses on sampling from the unit cube. Many problems, especially in computer graphics, are defined via quadrature over the unit triangle. 
Quasi-Monte Carlo methods for the triangle have been developed by \cite{pill:cool:2005} and by \cite{bran:colz:giga:trav:2013}.
This paper presents two QMC constructions in the triangle with a vanishing
discrepancy.  The first is a version of the van der Corput sequence customized to the
unit triangle.  It is an extensible digital construction that attains a discrepancy below $12/{\sqrt{N}}$. 
The second construction rotates an integer lattice
through an angle whose tangent is
a quadratic irrational number.
It attains a discrepancy of $O(\log(N)/N)$ which is the
best possible rate. Previous work strongly indicated that
such a discrepancy was possible, but no constructions were available.
Scrambling the digits of the first construction improves its accuracy
for integration of smooth functions.
Both constructions also yield convergent estimates for integrands that are Riemann integrable on the triangle without requiring bounded variation.
\\

\par\noindent
\textbf{Keywords} : Quasi-Monte Carlo Method; Discrepancy on Triangle; Quadrature error.
\end{abstract}

\section{Introduction}

The problem we consider here is numerical integration
over a triangular domain, using quasi-Monte Carlo (QMC) sampling.
Such integrals commonly arise in graphical rendering.
Classical quadrature methods find a set of points 
$\bsx_1,\dots,\bsx_N$ in the triangle
and weights $w_i\in \real$ so that
$\sum_{i=1}^Nw_if(\bsx_i)$ correctly integrates
a class of polynomials $f$.
\cite{lyne:cool:1994} give a survey.

Classical rules often do poorly on non-smooth
integrands. It is also difficult to estimate
error for them and there is little freedom to
choose $N$.
As a result, QMC sampling, which equidistributes
sample points through the domain of interest is attractive.

QMC sampling is well developed
for numerical integration
of functions defined on the unit cube $[0,1]^d$.
The quantity $\mu = \int_{[0,1]^d} f(\bsx) \rd \bsx  $ is
approximated by an equal weight rule
$\hat\mu_N=(1/N)\sum_{i=1}^N f(\bsx_i)$ for carefully chosen $\bsx_i\in[0,1]^d$.
The accuracy of QMC is customarily measured via
the Koksma-Hlawka inequality (see \cite{nied92}).
There
$|\hat\mu_N-\mu|\le D_N^*(\bsx_1,\dots,\bsx_N)\times V_\hk(f)$
where $D_N^*$ is the star discrepancy (a measure of non-uniformity)
of the sample points and $V_\hk$ is the total variation of $f$
in the sense of Hardy and Krause.

The usual approach to sampling these domains
is to apply a mapping $\phi$ from $[0,1]^d$ to the
domain $D$ of interest. The mapping is such that
if $\bsx\sim\dustd([0,1]^d)$ (uniform distribution)
then $\phi(\bsx)\sim\dustd(D)$.
There are typically several choices for such mappings
and the dimension $d$ of the cube is not necessarily
equal to the dimension of $D$.
With such a mapping in hand we may generate QMC
points $\bsx_i\in[0,1]^d$ and use $\phi(\bsx_i)$
as sample points in $D$.

Using this approach we may
estimate $\mu = \int_D g(\bsx)\rd\bsx$
by $(1/n)\sum_{i=1}^nf(\bsx_i)$ where
$f(\bsx) = g(\phi(\bsx))$.
The difficulty is that the composite function
$f = g\circ \phi$ may not be well suited
to QMC; it may have cusps or singularities
or discontinuities.  
These features may diminish the
performance of QMC. At a minimum, they make it more
difficult to analyze QMC's performance.

Recently \cite{bran:colz:giga:trav:2013} presented a
version of the Koksma-Hlawka inequality for the simplex.
They devised a measure of variation for the simplex
and a discrepancy measure for points in the simplex. But they
did not present a sequence of points with vanishing
discrepancy.

\cite{pill:cool:2005} also studied QMC integration
over the simplex. They mention that
the Koksma-Hlawka bound can be applied using the
discrepancy of the original points $\bsx_i$
and the variation of the composite function $g\circ \phi$,
but do not give conditions for that variation to be finite.
They also devised a measure of variation for functions
on the simplex, a corresponding discrepancy measure for points
inside the simplex, and a Koksma-Hlawka bound using
these two factors. But they did not obtain a link between
the cube discrepancy of their original points and the
simplex discrepancy of the image of those points under $\phi$.

Neither \cite{bran:colz:giga:trav:2013}
nor \cite{pill:cool:2005} provide a QMC construction for
the simplex with a vanishing discrepancy.
In this paper we present two constructions for points in
the triangle.  The first is an extensible digital construction
that mimicks the van der Corput sequence and exploits a recursive
partitioning of the triangle. The second resembles a hybrid
of lattice points \citep{sloanjoe} and the Kronecker construction
\citep{larc:nied:1993}. A rectangular grid of points is rotated
through a judiciously chosen angle and those that intersect
the triangle are retained. We combine theorems of
\cite{Chen2007662} and \cite{bran:colz:giga:trav:2013}
to show that our points have vanishing
discrepancy. This second construction has better discrepancy
but the digital one is extensible and is amenable to digital
scrambling among other things.

For both of these constructions, the discrepancy
of \cite{bran:colz:giga:trav:2013} vanishes as the number $N$
of points increases.  
The discrepancy of  \cite{pill:cool:2005} also vanishes.
We believe that these are the first constructions of points
in the triangle to which a Koksma-Hlawka inequality applies.

An outline of this paper is as follows.
Section~\ref{sec:back} presents results
from the literature that we need along with
notation to describe those results.
We show there that the discrepancy of 
\cite{pill:cool:2005} is no larger than
twice that of \cite{bran:colz:giga:trav:2013} so that
the former vanishes whenever the latter does.
In Section~\ref{sec:cons} we adapt the van der Corput
sequence from the unit interval to an arbitrary
triangle. The result is an extensible sequence.
We show that the parallelogram
discrepancy of \cite{bran:colz:giga:trav:2013}
is at most $12/\sqrt{N}$ when using the
first $N$ points of our triangular van der Corput sequence
and it is exactly $2/(3\sqrt{N})-1/(9N)$ when $N=4^k$.
Section~\ref{sec:trikron} develops a second explicit
construction. It rotates a scaled copy of $\ints^2$
through a carefully chosen angle, keeping only those points
that lie within a right angle triangle.  The resulting points
have parallelogram discrepancy $O(\log(N)/N)$ and retain
that discrepancy when mapped to an arbitrary
nondegenerate triangle.
Integration over a triangle is an important sub-problem
in computer graphics. But there the integrands are often
discontinuous and of infinite variation. Quasi-Monte
Carlo over the cube has vanishing error so long
as $f$ is Riemann integrable \citep{nied92}.
Section~\ref{sec:riem} shows that triangular van
der Corput points yield integral estimates with
vanishing error whenever the integrand is merely
Riemann integrable over the triangle.
Section~\ref{sec:disc}  has some final discussion.

We conclude this section by describing some
more of the literature.
\cite{fang:wang:1994} give volume preserving
mappings from the unit cube to the ball,
sphere and simplex in $d$ dimensions all of which
can be used to generate QMC samples in those
other spaces.
\cite{aist:brau:dick:2012} study QMC in the sphere.
\cite{pill:cool:2005} present $5$ different mappings
from the unit cube to the simplex.  Additionally they
consider an approach that embeds the simplex within
a cube and ignores any QMC points from the cube
that do not also lie in the simplex.
\cite{arvo:1995} gives a mapping for spherical triangles.
Further mappings are based on probabilistic
identities, such as those in \cite{devr:1986}.
These mappings are equivalent when applied to uniform
random inputs.  But they differ for QMC points. Some
are many-to-one and others have awkward Jacobians,
inevitable when mapping a region with 4 corners onto one with 3.
Discontinuities and singular Jacobians can
yield infinite variation \citep{variation} when
the integrand is viewed as a function on $[0,1]^d$.

\section{Background}\label{sec:back}

Here  we present some notation that we need.
Then we describe previous results.

The point $\bsx\in\real^d$ has components $x_j$
for $j=1,\dots,d$.
We abbreviate $\{1,2,\dots,d\}$ to $1{:}d$. 
The set $u\subseteq1{:}d$ has cardinality
$|u|$ and complement $-u \equiv\{j\in1{:}d\mid j\not \in u\}$.
The point $\bsx_u\in\real^{|u|}$ contains the components $x_j$ of
$\bsx$ for $j\in u$. 
Sometimes we combine components of two points to make a new
one. Given $\bsx$, $\bsy\in\real^d$ and
$u\subseteq1{:d}$, the hybrid $\bsx_u{:}\bsy_{-u}\in\real^d$ is the 
point $\bsz$ with $z_j= x_j$ for $j\in u$ and $z_j=y_j$ otherwise.
The point $\one$ is the vector of $d$ 1s.
Thus $\bsx_u{:}\one_{-u}$ is the point $\bsx$ after every $x_j$
for $j\in u$ has been replaced by $1$.

Some computations and expressions are simpler with
one triangle than they are with another.  
Let $\ptA$, $\ptB$, and $\ptC$ be three non-collinear points
in $\real^d$.
Those points define the non-degenerate triangle
$$\Delta(\ptA,\ptB,\ptC)
= \{ \omega_1\ptA+\omega_2\ptB+\omega_3\ptC
\mid \min(\omega_1,\omega_2,\omega_3)\ge 0, 
\omega_1+\omega_2+\omega_3=1\}.
$$
The simplex is usually defined via
with corners $(0,0,1)^\tran$ $(0,1,0)^\tran$, and
$(1,0,0)^\tran$. 
For some computations it is convenient to use
the equilateral triangle defined by
$\ptA = (0,0)^\tran$,
$\ptB = (1,0)^\tran$, and
$\ptC = (1/2,\sqrt{3}/2)^\tran$.
For some purposes we may scale the points so that our
triangle has unit area. At other times one scales the
triangle to have area equal to the number $N$ of points in a quadrature rule.
\cite{pill:cool:2005} used the right-angle triangle
\begin{align}\label{eq:tpillcool}
T_{\pc} = \Delta((0,0)^\tran,(0,1)^\tran ,(1,1)^\tran).
\end{align}
Our lattice construction uses
$\Delta((0,0)^\tran,(0,1)^\tran ,(1,0)^\tran)$.

\subsection{Discrepancy}

Here we define the notions of discrepancy
that we need, for quadrature problems over
a set $\Omega$.  We follow 
\cite{bran:colz:giga:trav:2013} in taking
$\Omega$ to be a bounded Borel subset of $\real^d$.
We use $\vol(\cdot)$ to denote $d$-dimensional Lebesgue measure.
If $\Omega$ is contained in a linear flat subset of $\real^d$
then we interpret volumes as Lebesgue measure with respect to the 
lowest-dimensional such linear flat.
To exclude uninteresting cases, we assume that $\vol(\Omega)>0$.
For $N\ge 1$, let $\cp = (\bsx_1,\dots,\bsx_N)$ be a list
of (not necessarily distinct) points in $\real^d$.
For a set $S\subset\real^d$, we let $A_N$ be the counting function,
$A_N(S;\cp) = \sum_{i=1}^N1_{\bsx_i\in S}$.  The signed discrepancy 
of $\cp$ at the measurable set $S\subset\real^d$ is
$$\delta_N(S;\cp,\Omega) 
= \vol(S\cap\Omega)/\vol(\Omega)-A_N(S;\cp)/N.$$
The signed discrepancy has a useful additive property.
If $S_1\cap S_2=\emptyset$, then
\begin{align}\label{eq:additive}
\delta_N(S_1\cup S_2;\cp,\Omega) = 
\delta_N(S_1;\cp,\Omega) +\delta_N(S_2;\cp,\Omega).
\end{align}
Also, $\delta_N(\emptyset;\cp,\Omega)=0$.
The absolute discrepancy 
of points $\cp$ for a class $\cs$ of measurable subsets
of $\Omega$ is
$$
D_N(\cs;\cp,\Omega)=
\sup_{S\in\cs} D_N(S;\cp,\Omega),
\quad\text{where}\quad
D_N(S;\cp,\Omega) = |\delta_N(S;\cp,\Omega)|.$$

For general $\Omega$ it is helpful to extend $\cp$
by all integer shifts, that is
by considering all $\bsx_i+\bsm\in\Omega$
for $i=1,\dots,N$ and $\bsm\in\ints^d$. Because $\Omega$ is bounded,
the extension still has finitely many points.  We define
the extended count
$$\bar A_N(S;\cp) = \sum_{\bsm\in\ints^d}\sum_{i=1}^N1_{\bsx_i+\bsm\in S}$$
and then take
\begin{align}\label{eq:extendeddiscrep}
\bar D_N(\cs;\cp,\Omega)= \sup_{S\in\cs}|\bar \delta_N(S;\cp,\Omega)|,
\end{align}
where $\bar \delta_N(S;\cp,\Omega) =
\vol(S\cap\Omega)/\vol(\Omega)-\bar A_N(S;\cp)/N$.
Notice that $\bar A_N$ is divided by $N$, and not
the number of extended points lying in $\Omega$.
When $\Omega$ is understood, we may simplify the
discrepancies to $D_N(\cs;\cp)$,  $\bar D_N(\cs;\cp)$.
Likewise $\cs$ can be omitted.

Standard quasi-Monte Carlo sampling \citep{nied92}  works with
$\Omega = [0,1)^d$ and takes for $\cs$ the set of 
anchored boxes $[0,\bsa)$ with $\bsa\in[0,1)^d$.
Then $D_N(\cs;\cp)$ above is the star-discrepancy
$D_N^*(\cp)$.
Now let a real-valued function $f$ be defined on $[0,1]^d$
(not just $[0,1)^d$) with variation
$V_\hk(f)$ in the sense of Hardy and Krause. Then
the Koksma-Hlawka inequality is
$$
\Bigl|\frac1N\sum_{i=1}^Nf(\bsx_i) - \int_{[0,1)^d}f(\bsx)\rd\bsx\Bigr|
\le D_N^*(\cp)V_\hk(f).
$$
If the needed derivatives are continuous, then
$$
V_\hk(f) = \sum_{u\subseteq{1:d},u\ne\emptyset}
\int_{[0,1]^{|u|}} 
\Bigl|
\frac{\partial^{|u|}}{\partial\bsx_u}
f(\bsx_u{:}\one_{-u})
\Bigr|\rd\bsx_u.
$$

\cite{bran:colz:giga:trav:2013} provide a Koksma-Hlawka
inequality for parallelepipeds.
Unlike the usual Koksma-Hlawka inequality, their variation
measure sums integrals over all faces of all dimensions
of the parallepiped.
They then represent the indicator function of a simplex
defined by $d+1$ corner points 
as the weighted sum of indicators of $d+1$ parallepipeds.
The $j$'th parallepiped has one vertex at the $j$'th
corner of the simplex and its $d$ defining vectors extend
from that $j$'th vertex to the other $d$ corners.
Their non-negative weighting function varies spatially, summing to
$1$ within the simplex.
They then obtain a Koksma-Hlawka inequality for the simplex
based on their inequality for parallepipeds.

Here we present their discrepancy measure for the
case of a triangle with corners $\ptA$, $\ptB$ and $\ptC$.
For real values $a$ and $b$, let
$\calt_{a,b,\ptC}$ be the parallelogram
defined by the point $\ptC$ with vectors
$a(\ptA-\ptC)$ and $b(\ptB-\ptC)$.
One such parallelogram is illustrated in Figure~\ref{fig:brandotabc}
where it has vertices $C$, $D$, $F$ and~$E$.
Let
\begin{align}\label{eq:trapc}
\cs_C & = 
\{ \calt_{a,b,\ptC}\mid 0<a<\Vert\ptA-\ptC\Vert, 0<b<\Vert\ptB-\ptC\Vert\}
\end{align}
and define $\cs_A$ and $\cs_B$ analogously.
Then the parallelogram discrepancy of points $\cp$
for $\Omega = \Delta(\ptA,\ptB,\ptC)$ is
$$
D_N^{\pr}(\cp;\Omega)
 = D_N(\cs_\pr;\cp,\Omega),\quad\text{for}\quad
\cs_\pr  = \cs_A\cup\cs_B\cup\cs_C.
$$

\begin{figure}[t]
\center
\includegraphics[scale = 0.45]{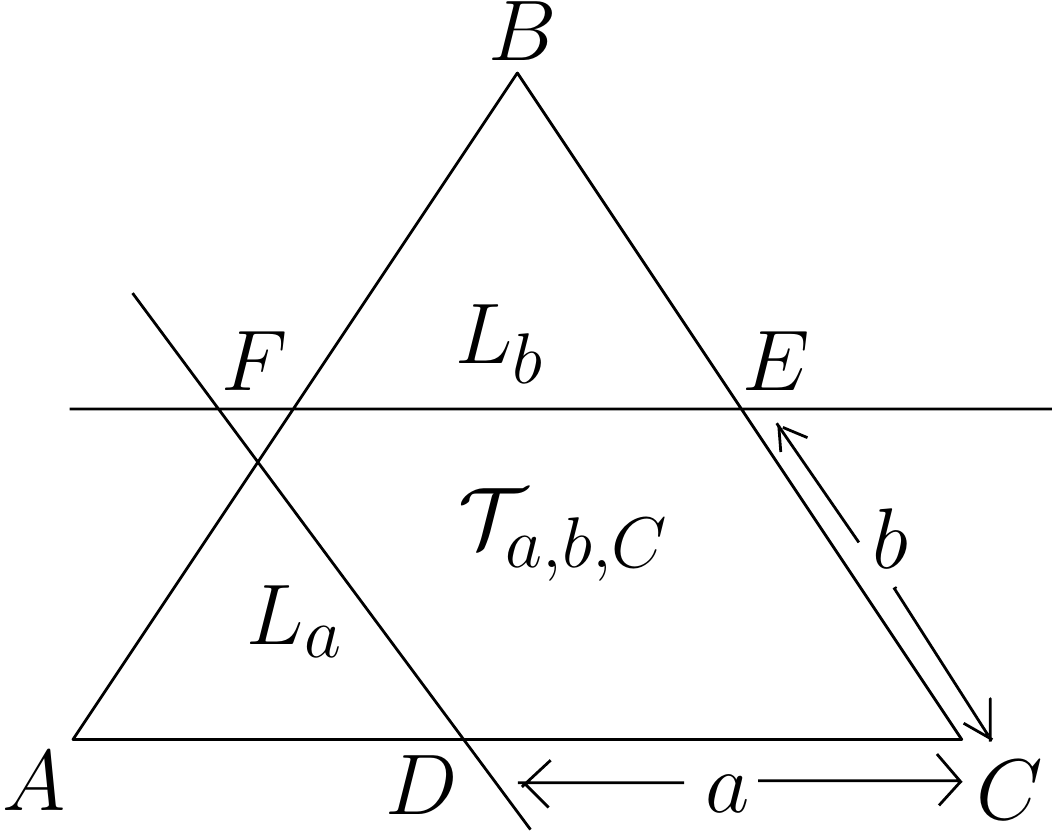}
\caption{
\label{fig:brandotabc}
The construction of the parallelogram $\mathcal{T}_{a,b,C}=CDFE$.}
\end{figure}

\cite{pill:cool:2005} also define a discrepancy
for simplices.  For simplices with three vertices,
their $\Omega$ is the triangle $T_{\pc}$
from~\eqref{eq:tpillcool}.
They measure discrepancy using anchored boxes,
studying 
\begin{align}
D_N^\pc(\cp;T_\pc) =D_N(\cs_I,\cp,T_\pc)\quad\text{where}\quad
\cs_I &= \{ [0,\bsa)\mid \bsa\in[0,1)^2\}.
\end{align}

\begin{lemma}\label{lem:discineq}
Let $T_\pc$ be the triangle from~\eqref{eq:tpillcool}
and for $N\ge1$, let  $\cp$ be the list of
points $\bsx_1,\dots,\bsx_N\in T_\pc$.
Then 
$D_N^\pc(\cp,T_\pc) \le 2D_N^\pr(\cp,T_\pc)$
and $\bar D_N^\pc(\cp,T_\pc) \le 2\bar D_N^\pr(\cp,T_\pc).$
\end{lemma}
\begin{proof}
Let $[0,a_1)\times[0,a_2)$ be an anchored box in $[0,1]^2$.
We may write it as the difference $[0,a_1)\times[0,1) - [0,a_1)\times[a_2,1)$
of sets in $\cs_C$ taking $C$ to be the vertex $(0,1)^\tran$
of $T_\pc$.
Then $D_N^\pc(\cp;T_\pc) \le 2D_N(\cs_C,\cp,T_\pc)\le 2D_N^\pr(\cp,T_\pc)$.
The same argument holds for $\bar D_N$.
\end{proof}

From Lemma~\ref{lem:discineq} we see that a sequence
with vanishing parallel discrepancy will also have
vanishing discrepancy in the sense of Pillands and Cools.

\subsection{Koksma-Hlawka}
\cite{bran:colz:giga:trav:2013}
define a corresponding variation measure that
we will call $V_{\pr}(f)$.  The specialization of
this measure to the triangle appears on the last
page of their article.  Rather than reproduce it here
we remark that it is a weighted sum of some integrals over the triangle,
some integrals over the edges of the triangle, and function evaluations
at the corners of the triangle.  The corner evaluations are absolute
values of $f$ at those corners.  The edge integrals are 
averages of $|f|$ plus the absolute value of the interior
directional derivative of $f$ along that edge.
The integrand on the whole triangle sums the absolute value of
$3f$ as well as first order directional derivatives of $2f$
and second order directional derivatives of $f$.
The entire sum is multiplied by a constant $C_2>0$ known to be finite.
Note that their variation is positive for (nonzero)
constant functions.

The numerical treatment of sample points $\bsx_i$
is different when those points
are on the boundary of $\Omega$.
Let $\Omega$ be a closed polytope in $\real^d$
not lying in a flat of dimension $d-1$ or less.
Then we define the weight function
$$w_\Omega(\bsx)
=\begin{cases}
0, & \bsx\not\in\Omega,\\
1, & \bsx\in\text{the interior of $\Omega$},\\
2^{k-d} & \bsx\in\text{a $k$-dimensional face of $\Omega$.}
\end{cases}
$$
The integer $k$ is understood to be the smallest dimension
of any face of $\Omega$ that contains $\bsx$.
When $\Omega$ lies in a lower dimensional flat
we work instead with the relative interior of $\Omega$
and similarly replace $d$ by the smallest containing
dimension.  
Given $\cp$ with points $\bsx_1,\dots,\bsx_N$ and
a function $f$ on $\Omega$ we define
\begin{align}\label{eq:weightedf}
\sum^*_{\cp,\Omega} f 
= \sum_{i=1}^N\sum_{\bsm\in\ints^d}f(\bsx_i+\bsm)
w_\Omega(\bsx_i+\bsm).
\end{align}

\begin{theorem}\label{thm:brando}
Let $\cp$ be a list of $N$ points in $\real^d$
and let $\Omega = \Delta(\bsa,\bsb,\bsc)$ be a non-degenerate triangle
in $\real^d$. Then
$$
\Bigl|
\int_{\Omega} f(\bsx)\rd\bsx
-\frac1N\sum^*_{\cp,\Omega} f 
\Bigr|
\le \bar D_N^\pr(\cp,\Omega) \times V_\pr(f).
$$
\end{theorem}
\begin{proof}
This is Theorem 3.2 of 
\cite{bran:colz:giga:trav:2013} specialized to the triangle.
\end{proof}

\subsection{Transformations}
Given two non-degenerate triangles $\Delta(\ptA,\ptB,\ptC)\subset\real^D$
and $\Delta(a,b,c)\subset\real^d$ there is a linear
mapping $M$ from $\real^D$ to $\real^d$
with $M\ptA=a$, $M\ptB=b$, and $M\ptC=c$.
If we make a transformation of $\bsx_i$ to $M\bsx_i$
and call the resulting points $M\cp$,
then
$D_N^\pr(M\cp;\Delta(a,b,c))
=D_N^\pr(\cp;\Delta(\ptA,\ptB,\ptC))$.
The same does not hold for $D_N^\pc$ because
a linear transformation can map anchored boxes onto
parallelepipeds.

\section{Triangular van der Corput points}\label{sec:cons}

The digital construction we use works by
lifting the construction of~\cite{vand:1935:I} 
from the unit interval to the triangle.
In van der Corput sampling of $[0,1)$
the integer $i\ge0$ is written in the integer base $b\ge2$
as $\sum_{k\ge1}d_kb^{k-1}$ where $d_k=d_k(i)\in\{0,1,\dots,b-1\}$.
Then $i$ is mapped to $x_i = \sum_{k\ge 1}d_kb^{-k}$.
The points $x_1,\dots,x_{n}\in[0,1)$ have a discrepancy
of $O(\log(n)/n)$. 

For triangular van der Corput points,
we first partition the triangle into $4$ congruent
subsets as shown by the leftmost panel
in Figure~\ref{fig:subdiv}. 
We assign base $4$ digits $0$ through $3$ to these 
subtriangles with $0$ in the center and
the others subject to an arbitrary choice.
Each such triangle can be partitioned again
in a similar manner as shown in the second panel.

\begin{figure}[t]
\centering
\includegraphics[width=\hsize]{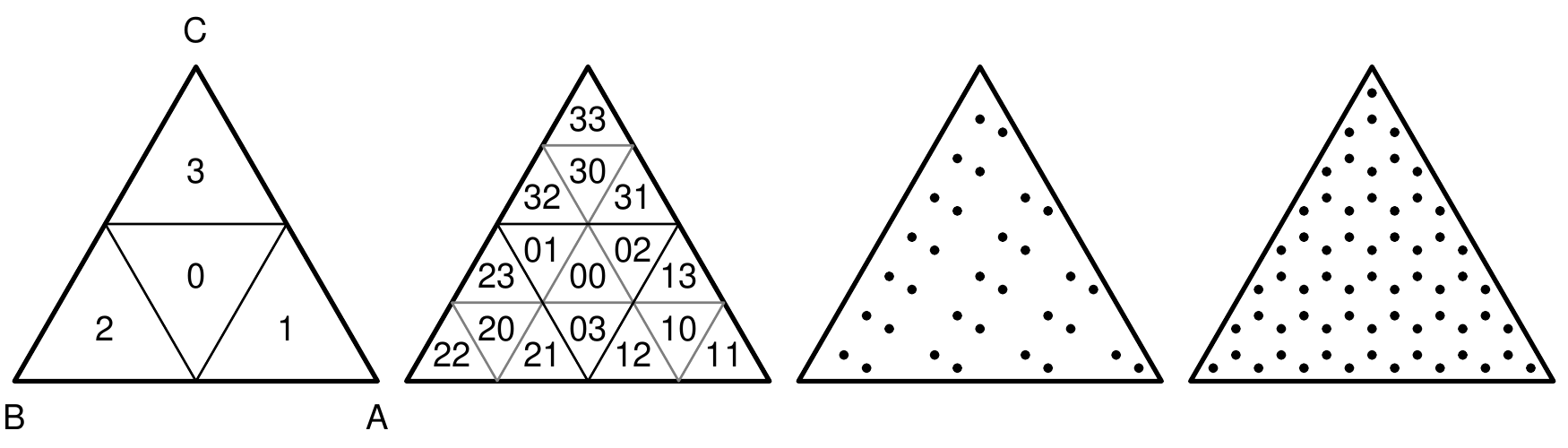}
\caption{\label{fig:subdiv}
A labeled subdivision of $\Delta(A,B,C)$ into $4$ and then
$16$ congruent subtriangles. Next are the first
$32$ triangular van der Corput points followed
by the first $64$.
}
\end{figure}

We write the integer $i>0$ in a base
$4$ representation $i=\sum_{k=1}^{K_i}d_k4^{k-1}$
where $d_k=d_k(i)\in\{0,1,2,3\}$ and $K_i = \lceil\log_4(i)+1\rceil$.
Given a triangle $T$, we map the integer $i$ to the point
$f_T(i)\in T$ as follows.
First we identify the subtriangle of $T$ corresponding to $d_1$.
Call it  $T(d_1)$. 
Then we get the subtriangle $T(d_1,d_2)=(T(d_1))(d_2)$ corresponding
to digit $d_2$ within $T(d_1)$, and so on.
This process maps the integer $i$ to the triangle
$T(d_1,d_2,\dots,d_{K_i})$. 
The point $f_T(i)$ is the center point of
triangle $T(d_1,d_2,\dots,d_{K_i})$. The center of the
triangle is the arithmetic average of its vertices.
The triangle $T(d_1,d_2,\dots,d_{K_i},0,0,\dots,0)$
also has center $f_T(i)$, and as we increase the number
of zeros beyond $d_{K_i}$, the three corners of the
resulting triangle all converge to $f_T(i)$.
For $i=0$, our convention is that $f_T(0)$ is the center of
the original triangle $T$.

We have not yet formally specified which subtriangle
of $T(d_1)$ we mean by $T(d_1,d_2)$ when $d_2\ne 0$.
To make this precise,
let $T=\Delta(A,B,C)$  be an arbitrary triangle.
Then for $d\in\{0,1,2,3\}$ the subtriangle of $T$ is
$$
T(d) = 
\begin{cases}
\Delta\bigl( \frac{B+C}2,\frac{A+C}2,\frac{A+B}2\bigr), & d=0\\[0.5ex]
\Delta\bigl( A, \frac{A+B}2,\frac{A+C}2\bigr), & d=1\\[0.5ex]
\Delta\bigl( \frac{B+A}2, B,\frac{B+C}2\bigr), & d=2\\[0.5ex]
\Delta\bigl( \frac{C+A}2, \frac{C+B}2,C\bigr), & d=3.
\end{cases}
$$
This pattern is followed in Figure~\ref{fig:subdiv}.
If we represent the triangle $T$ by a vector of the three
corner points $A$, $B$, and $C$, then
$T(1) = (A+T)/2$ componentwise, and similarly
$T(2) = (B+T)/2$, $T(3)=(C+T)/2$ and
$T(0) = (A+B+C)/2-T/2$.

This construction defines an infinite sequence
of $f_T(i)\in T$ for integers $i\ge 0$.
For an $n$ point rule, 
take $\bsx_i=f_{T}(i-1)$ for $i=1,\dots,n$.

This triangular van der Corput sequence has several
desirable properties.
First, it is extensible.  If we have sampled $n$
points and find that we need $m$ more we
simply take the next $m$ points in the sequence.
Second, it is balanced.  
If $n=4^k$ then we get the centers of a symmetric 
triangulation as shown by the final panel in Figure~\ref{fig:subdiv}.
If our sample is not a multiple of $4^k$, we still have
reasonable balance, as illustrated by the third panel in
Figure~\ref{fig:subdiv}. There are $32$ points of which
the second $16$ points fall into gaps left by the first $16$
points.

\subsection{Discrepancy of triangular van der Corput points}\label{sec:disc}

In this subsection, we state and prove some results
on the parallel discrepancy of the triangular
van der Corput points.
That discrepancy is the same for
any triangle.  We will work with an equilateral
triangle $\etri$ of unit area so that discrepancy
calculations reduce to computing areas and
counting points. Such a triangle
has sides of length $\ell = 2/\sqrt[4]{3}$.

Our discussion of these points revolves around
a standard decomposition of $\etri$ into $N=4^k$
subtriangles of area $1/N$. These subtriangles are similar
to $\etri$ and have sides parallel to those of $\etri$.
The first two panels in Figure~\ref{fig:subdiv} depict
such decompositions into $4$ and $16$ subtriangles, respectively.
The points $\bsx_i=f_\etri(i-1)$ for $1\le i\le 4^k$ are at
the centroids of these subtriangles.
When we plot $\etri$ with a horizontal base below its peak,
then $2^k(2^k+1)/2$ of the subtriagles will also be
pointing up that way. We call these upright subtriangles.
We call the remaining $2^k(2^k-1)/2$ subtriangles
inverted subtriangles.

For our purposes here a line segment `touches' a triangle if it
intersects an \emph{interior} point of that triangle, splitting
it into two subsets of positive area. 

\begin{theorem}\label{thm:vdcdiscrep}
For an integer $k\ge0$ and non-degenerate
triangle $\Omega=\Delta(A,B,C)$, let $\cp$ consist of $\bsx_i = f_\Omega(i-1)$
for $i=1,\dots,N=4^k$. Then
$$
D_N^\pr(\cp;\Omega) = \begin{cases}
\dfrac79, & N=1\\[1ex]
\dfrac2{3\sqrt{N}}-\dfrac1{9N},&\text{else.}
\end{cases}
$$
\end{theorem}

The proof of Theorem~\ref{thm:vdcdiscrep} requires
consideration of numerous subcases.  We defer it to
Section~\ref{sec:proveit}.
For $N=1$, the maximal discrepancy is attained
by a parallelogram 
just barely including the center point
and holding $2/9$ of the area. It has positive signed discrepancy.
For $N=4^k>1$, the maximal discrepancy is attained 
(in the limit) by
the trapezoid just barely excluding all $\sqrt{N}$
van der Corput points in the `bottom row' of $\etri$ and the
signed discrepancy is negative. 
The same limit is also attained in the limit
for a sequence of trapezoids having positive signed discrepancy.

\begin{theorem}\label{thm:any4tothek}
Let $\Omega$ be a nondegenerate triangle, and
let $\cp$ contain points $\bsx_i = f_\Omega(s+i-1)$, 
$i=1,\dots,N=4^k$, for a starting integer $s\ge1$ 
and an integer $k\ge 0$. Then
\[D_N^\pr(\cp;\Omega) \leq \frac{2}{\sqrt{N}} - \frac{1}{N}.
\]
\end{theorem}
\begin{proof}
A set $S\in \cs_C$ can be written $S=\ct_{a,b,C}\cap \Omega$.
Let $T_j$ be the interiors of the subtriangles of $\Omega$ for $j=1,\dots,N$
and then let $T_0=\Omega\, \backslash \cup_{j=1}^NT_j$.
Now define $S_{j} = S\cap T_j$, $j=0,1,\dots,N$.
Then $\delta_N(S) = \sum_{j=0}^N\delta_N(S_j)$,
from~\eqref{eq:additive}.
Because the $\bsx_i$ are all interior points of their respective
subtriangles, we have $\delta_N(S_0)=0$.  If the boundary
of $\ct_{a,b,C}$ does not touch $S_j$ for $1\le j\le N$
then $\delta_N(S_j)=0$ too.  Otherwise $-1/N\le \delta_N(S_j)\le 1/N$.
Therefore $D_N(S;\cp)\le m/N$ where $m$ is the number of subtriangles
touching a boundary line of $\cs_{a,b,C}$. No such trapezoid can
have a boundary touching more than $2\sqrt{N}-1$ subtriangles.
Therefore
$D_N(\cs_C;\cp)\le (2\sqrt{N}-1)/N$ and since the same holds
for $\cs_A$ and  $\cs_B$, the theorem follows.
\end{proof}

\begin{theorem}\label{thm:anyn}
Let $\Omega$ be a non-degenerate triangle and, for integer $N\ge 1$, let
 $\cp = (\bsx_1,\dots,\bsx_N)$, where $\bsx_i = f_\Omega(i-1)$.
Then
$$D_N^\pr(\cp;\Omega) \le 12/\sqrt{N}.$$
\end{theorem}
\begin{proof}
Let $N = \sum_{j = 0}^K a_j 4^j$ for integer $K>0$, with $a_K\ne 0$.
Let $\cp_j^l$ 
denote a set of $4^j$ consecutive points 
from $\cp$, for $l = 1,\ldots, a_j$ and $j\le K$.
These $\cp_j^l$ can be chosen to partition the $N$ points $\bsx_i$. 
Fix any $S \in \cs_\pr$. Then,
\begin{align*}
\delta_N(S;\cp) =
\frac1N\sum_{j=0}^K\sum_{l=1}^{a_j} 4^j \delta(S;\cp_j^l) .
\end{align*}
Therefore from Theorem~\ref{thm:vdcdiscrep},
$$
D_N(S;\cp)=|\delta_N(S;\cp)| \le
\frac1N\sum_{j=0}^K\sum_{l=1}^{a_j} 4^j\Bigl( \frac{2}{2^j}-\frac1{4^j}\Bigr)
\le\frac1N\sum_{j=0}^Ka_j(2^{j+1}-1).
$$
Because $a_j\le 3$,
\begin{align*}
D_N(S;\cp) & \le
\frac3N\bigl( 2(2^{K+1}-1)-(K+1)\bigr)
\le \frac{12\times2^K}N
\end{align*}
and then $K\le \log_4(N)$,
gives $D_N(S;\cp) \le {12}/\sqrt{N}$.
Taking the supremum over $S\in \cs_\pr$ yields the result.
\end{proof}
Note that this bound can be improved by subtracting a multiple of
$\log(N)/N$ but that does not affect the rate.

If we apply the nested uniform
digit scrambling of  \cite{rtms} to the base $4$
digits of $i-1$, then $\bsx_i$ for $i=1,\dots,N=4^k$
are independent
and uniformly distributed within their subtriangles.
In that case, if $f$ has bounded first derivative
on $\etri$ then $$\e\Bigl( \Bigl( 
\frac1N\sum_{i=1}^nf(\bsx_i)
-\int_\etri f(\bsx)\rd\bsx\Bigr)^2\Bigr) = O\Bigl( \frac1{N^2}\Bigr)$$
because the subtriangles have diameter $O(1/\sqrt{N})$.

\subsection{Proof of Theorem~\ref{thm:vdcdiscrep}}\label{sec:proveit}
The parallel discrepancy is the same for all nondegenerate
triangles, so we will work with $\Omega=\etri$.
By symmetry of the construction
$$D_N^\pr(\cp;\Omega)
= D_N(\cs_A;\cp,\Omega)
= D_N(\cs_B;\cp,\Omega)
= D_N(\cs_C;\cp,\Omega),$$
and so it suffices to study $D_N(\cs_C;\cp,\etri)$.

The sets in $\cs_C$ are of the form
$S_{a,b}\equiv\ct_{a,b}\cap\etri$
where $0<a\le\ell$ and $0<b\le \ell$,
as depicted in Figure~\ref{fig:brandotabc}.
The trapezoid $\etri = CDFE$ has a horizontal upper
boundary line segment $L_b\equiv DF$ and a lower slanted boundary 
line segment $L_a\equiv EF$.

The case with $N=1$ can be solved easily.  
It corresponds to the infimal area of a parallelogram
containing the centroid of $\etri$. From here on we
assume $N=4^k$ for $k\ge1$.

We will use the decomposition
of $\etri$ into $N=4^k$ congruent equilateral subtriangles each of
area $1/N$.
Of these, there are $2^k(2^k+1)/2$ upright
subtriangles, $2^k(2^k-1)/2$ inverted triangles
and $\cp$ places one point at the centroid of
these $4^k$ subtriangles.

Recall that a line segment `touches' a triangle if it
intersects an interior point of that triangle.
We say that the line segment `crosses'
a triangle if touches it and also intersects two points on the
boundary of the triangle.
If neither $L_a$ nor $L_b$ touch $\etri$, then
$\delta_N(S_{a,b};\cp)=0$.
If $L_b$ touches $\etri$ and $L_a$ does not, then
$S_{a,b}$ is the subset of $\etri$ below a horizontal
line and we easily find that the greatest discrepancy
for this case is 
\begin{align}\label{eq:emptytrap}
\sup_{0\le b<\ell}D_N(S_{\ell,b}) = 
\dfrac2{3\sqrt{N}}-\dfrac1{9N}
\end{align}
attained when $L_b$ passes just below the centroids
of the bottom row of upright triangles.  Such a line
contains $0$ points of $\cp$ and its volume is given
by~\eqref{eq:emptytrap}.
By symmetry $\sup_{0\le a<\ell}D_N(S_{a,\ell})$ takes the
same value.

It remains to consider the
case where both $L_a$ and $L_b$ touch $\etri$.
In this case, to maximize discrepancy, the 
horizontal line $L_b$ must either pass
just above a row of upright subtriangle's centroids, just below
such a row, or just above or below a row of inverted 
subtriangle's centroids. Similarly the slanted line $L_a$
must pass just left or just right of a slanted row of centroids,
or else discrepancy can be increased.

There are 4 cases. The intersection of $L_a$ and $L_b$
could be inside an upright subtriangle, inside
an inverted subtriangle, outside of $\etri$
touching two disjoint bands of subtriangles, or outside
of $\etri$ touching two bands of one or more subtriangles
that share an upright subtriangle.
These cases are illustrated in Figure~\ref{fig:options}.

\begin{figure}
\centering
\includegraphics[width=.8\hsize]{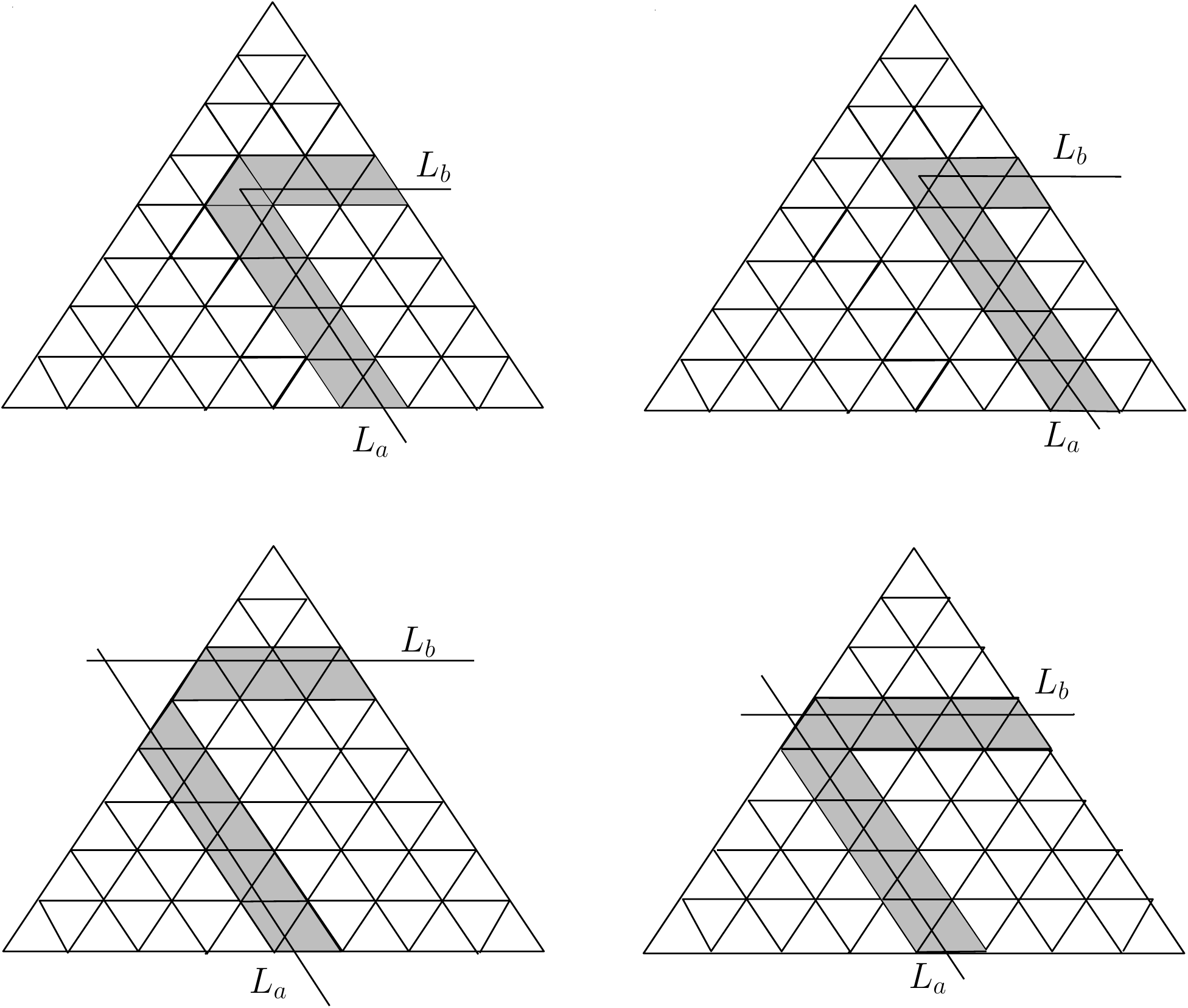}
\caption{\label{fig:options}
This figure illustrates the four cases that can
arise when both $L_a$ and $L_b$ touch $\etri$.
}
\end{figure}

As in Theorem~\ref{thm:anyn}, the signed discrepancy
$\delta_N(\cdot)$ can be summed over the subtriangles.
The cases in Figure~\ref{fig:options} include subtriangles
touched by $0$, $1$ or $2$ of the boundary lines.
A subtriangle  touched by $0$ boundary lines does not
contribute to the discrepancy.

Suppose that the upright subtriangle $T$ is crossed
by one horizontal line passing just above the
centroid of an inverted triangle to the left or right of $T$.
Referring to Figure~\ref{fig:trap} we see that
$8/9$ of the area of that upright triangle is below
the line as is its one point. As a result, the signed discrepancy
contribution $\delta_N(\cdot)$  for that subtriangle
is $1-8/9=1/9$ of the area of this triangle, that is
$1/(9N)$.
Similarly, the portion of an inverted subtriangle 
below that line has $4/9$ of the area and also the one and only point,
for a signed discrepancy contribution of $5/(9N)$.
These two facts are recorded in the first row
of Table~\ref{tab:triangledisc}.  The three other relevant
horizontal lines are also summarized in Table~\ref{tab:triangledisc}.

\begin{figure}
\centering
\includegraphics[width=.8\hsize]{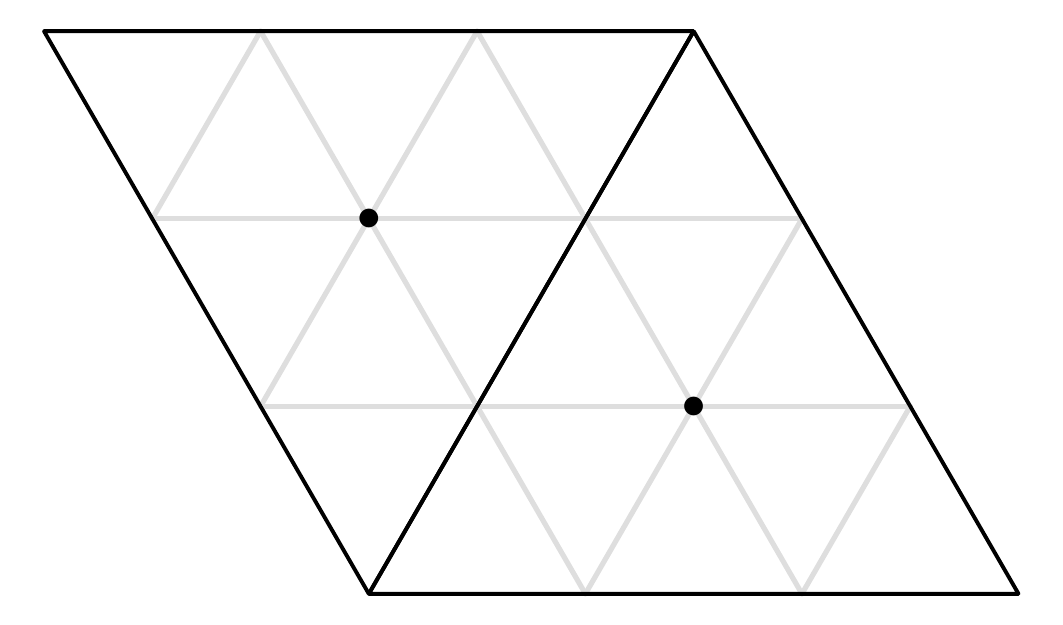}
\caption{\label{fig:trap}
This figure shows a trapezoid made up of one
upright subtriangle and one inverted subtriangle.
Each subtriangle has area $1/N$ and contains $1$
point of $\cp$ at its centroid, as shown.
}
\end{figure}

\begin{table}
\newcommand{\spx}{.8ex}
\centering
\begin{tabular}{c|ccccccc}
& \multicolumn{3}{c}{Upright} 
& \multicolumn{3}{c}{Inverted}
& \multicolumn{1}{c}{Total}\\
Horiz. Line& pts & vol & $\delta_u$ 
& pts & vol& $\delta_i$ 
&$\delta_u+\delta_i$\\[\spx]
\hline\\[-1.8ex]
Inv $+$ & $1$ &$8/9$ & $\phm1/9$ & $1$ & $4/9$ & $\phm5/9$ & $\phm2/3$
\\[\spx]
Inv $-$ & $1$ & $8/9$ & $\phm1/9$ & $0$ & $4/9$ & $-4/9$ & $-1/3$
\\[\spx]
Upr $+$ & $1$ & $5/9$ & $\phm4/9$& $0$ & $1/9$ & $-1/9$& $\phm1/3$
\\[\spx]
Upr $-$ & $0$& $5/9$ & $-5/9$& $0$ & $1/9$ & $-1/9$& $-2/3$
\\[\spx]
\hline
\end{tabular}
\caption{\label{tab:triangledisc}
Signed discrepancies for
subtriangles crossed horizontally by $L_b$
and not touched by $L_a$.
Line Inv$+$ passes
just above the centroid of the inverted subtriangle, Inv$-$
passes just below it. Upr$\pm$ are similarly defined with
respect to the centroid of the upright subtriangle.
For each subtriangle we record the number of centroid
points and the fraction of its volume below each line
as well as $N$ times the signed discrepancy contribution of that subtriangle
and the total signed discrepancy of the trapezoid they form.
}
\end{table}

Also in that table, we see the total discrepancy
of two triangles, one upright and one inverted,
when they are both crossed by the same horizontal line.
These subtrapezoids play an important role in the analysis.
The signed discrepancy contribution of a subtrapezoid can be has high
as $2/(3N)$ when the line crosses just above the
centroid of the inverted triangle, and as low as $-2/(3N)$
when it passes just below the centroid of the
upright triangle.  
The same discrepancies hold for triangles crossed
by  the slanted line $L_a$ intersecting the base of $\etri$
at distance $a$ from~$C$, where as before $\pm$ indicates
values of $a$ just barely including/excluding
a subtriangle's centroid.

Now consider a subtriangle touched by both lines
$L_a$ and $L_b$.
We can see from Figure~\ref{fig:trap}
that if an inverted subtriangle is touched
by both lines then they must have met at its centroid.
The signed discrepancy from that triangle is then $8/(9N)$
if both lines included the centroid and $-1/(9N)$ 
if either excluded it.

\begin{table}
\newcommand{\spx}{2.4ex}
\centering
\begin{tabular}{c|cccccc}
& \multicolumn{5}{c}{Slanted line $L_a$}\\
Horiz. Line& Inv$+$ & Inv$-$ & Up$+$ & Up$-$ & Right Trap.\\[\spx]
\hline\\[-.8ex]
Inv$+$ &$\phm\dfrac2{9N}$ & $\phm\dfrac2{9N}$ & 
     $\phm\dfrac5{9N}$ & $-\dfrac4{9N}$ &$\dfrac{2}{3N}$ 
\\[\spx]
Inv$-$ & $\phm\dfrac2{9N}$ & $\phm\dfrac2{9N}$ & $\phm\dfrac{5}{9N}$
&$-\dfrac4{9N}$&$-\dfrac{1}{3N}$ \\[\spx]
Upr$+$ & $\phm\dfrac5{9N}$ & $\phm\dfrac5{9N}$& $\phm\dfrac7{9N}$ 
&$-\dfrac2{9N}$&$\dfrac{1}{3N}$ 
\\[\spx]
Upr$-$ & $-\dfrac4{9N}$ & $-\dfrac4{9N}$& $-\dfrac2{9N}$ 
&$-\dfrac2{9N}$&$-\dfrac{2}{3N}$ \\[\spx]
\hline\\[-.4ex]
Lower Trap. & $\dfrac2{3N}$ & $-\dfrac1{3N}$& $\dfrac1{3N}$ 
&$-\dfrac2{3N}$& \\[\spx]
\hline
\end{tabular}
\caption{\label{tab:upright}
Signed discrepancies for an upright
subtriangle $T$ touched by $L_a$ and $L_b$.
Rows designate $4$ relevant horizontal lines,
columns the slanted lines.
The main table
shows the signed discrepancy of $T$. The
rightmost column shows the signed discrepancy
of trapezoids to the
right of $T$. The bottom row shows the signed discrepancy
of trapezoids below $T$.
}
\end{table}

If both boundary lines touch an upright subtriangle $T$
those lines can meet just above or just below $T$'s centroid,
or they can meet just above or just below the centroid
of an inverted subtriangle to the left of $T$.
Table~\ref{tab:upright} enumerates the cases along with
their signed discrepancies, the signed discrepancies
of any trapezoids to the right of $T$, and the
signed discrepancies of trapezoids below (and right) of~$T$.

Now we consider our four cases.
First, if $L_a\cap L_b$ is in an upright subtriangle $T$
then $T$ is the only subtriangle touched by two 
lines. The total signed discrepancy is that from
within $T$ together with as many as $\sqrt{N}-1$
trapezoids. For $N=1$ there were none of
these trapezoids, but for
$N\ge4$, at least $3$ trapezoids can contribute. 
Referring to Table~\ref{tab:upright}, we see
that maximizing the contribution from trapezoids
will maximize the discrepancy irrespective
of the signed discrepancy from $T$. We maximize
discrepancy by finding the largest possible
$|\delta_N(T)|$ for which either the trapezoids
in its row or column have $\delta_N=\pm2/(3N)$
with a sign matching $\delta_N(T)$.
The result using $\sqrt{N}-1$ such trapezoids gives
discrepancy
$$
\delta_N(S_{a,b}) = \frac{5}{9N} + (\sqrt{N}-1)\frac{2}{3N}
= 
\frac{2}{3\sqrt{N}}-\frac{1}{9N},
$$
tying the discrepancy~\eqref{eq:emptytrap} from the
large empty region below all the $\bsx_i$.

The second case has $L_a\cap L_b$ in an inverted subtriangle.
There is always an upright subtriangle to the right of
an inverted one, both of those subtriangles are
touched by $L_a$ and $L_b$, and no others are touched by two of
these lines.
The inverted triangle has signed discrepancy 
$-1/(9N)$ or $8/(9N)$ and
the upright triangle to its right has signed discrepancy $2/(9N)$
for a total of $1/(9N)$ or $10/(9N)$.
There can be as many as $\sqrt{N}-2$
parallelogram pairs contributing to the total
discrepancy which cannot therefore exceed
$$
\frac{10}{9N} + (\sqrt{N}-2)\frac{2}{3N} = \frac{2}{3\sqrt{N}}-\frac{2}{9N}.
$$
As a result, the second case cannot maximize discrepancy.

The third case has $L_a$ and $L_b$ intersecting
outside $\etri$ and touching
two bands of parallelograms that intersect in one
upright triangle $T$. The greatest possible discrepancy
here arises from $\sqrt{N}-1$ trapezoids and one upright
triangle.  
This is the same configuration as in case 1 and hence
cannot exceed~\eqref{eq:emptytrap} either.

The fourth and final case has $L_a$ and $L_b$ touching
two bands of parallelograms that don't intersect.
As a result there are at most $\sqrt{N}-2$ parallelograms
contributing to the discrepancy along with $2$
upright triangles touched by one line each.
The greatest absolute discrepancy attainable this way is thus
$$
(\sqrt{N}-2)\frac{2}{3N} + 2\times \frac{4}{9N}
=\frac{2}{3\sqrt{N}} -\frac{4}{9N}.
$$
Having exhausted the cases, we conclude that for $N=4^k>1$,
$D_N^\pr(\cp)=(2/3\sqrt{N})-1/(9N)$.
$\Box$

\section{Triangular Kronecker Lattices}\label{sec:trikron}
In this section we use Theorem 1 of \cite{Chen2007662} to construct points in the triangle with a parallel discrepancy of $O({\log (N)}/{N})$. The construction is through a suitably scaled copy of the lattice $\mathbb{Z}^2$
rotated through an angle.
The chosen angle makes tangents of certain angles badly
approximable in the same way that Kronecker sequences 
use badly approximable numbers for sampling of the unit cube 
\citep{larc:nied:1993}.
We begin with some definitions.

\begin{definition}
A real number $\theta$ is said to be \textit{badly approximable} if there exists a constant $c > 0$ such that $n || n\theta|| > c $ for every natural number $n \in \mathbb{N}$ and $|| \cdot ||$ denotes the distance from the nearest integer. 
\end{definition}

\begin{definition}
Let $a$, $b$, $c$ and $d$ be integers with
$b\ne0$, $d\ne0$ and $c>0$, where $c$ is not a perfect square.
Then $\theta = (a+b\sqrt{c})/d$ is a
\textit{quadratic irrational number}.
\end{definition}

Quadratic irrational numbers have a periodic repeating
continued fraction representation, and they
are badly approximable
\citep{hens:2006}.


Let $\Theta = \{\theta_1,\dots,\theta_k\}$ be a set
of $k\ge1$ angles in $[0,2\pi)$.
Then let $\ca(\Theta)$ be the set of convex polygonal
subsets of $[0,1]^2$ whose sides make an angle of
$\theta_i$ with respect to the horizontal axis.
Theorem 1 of \cite{Chen2007662} says that there
exists a constant $C_\Theta<\infty$ such that for
any integer $N>1$ there exists a list $\cp=(\bsx_1,\dots,\bsx_N)$
of points in $[0,1]^2$ with
\begin{align}\label{eq:lowdiscpts}
D_N(\ca(\Theta);\cp,[0,1]^2) < C_\Theta \log(N)/N.
\end{align}
Their proof of Theorem 1 relies on this lemma:

\begin{lemma}\label{lem:chen2.2}
Suppose that the angles $\theta_1,\dots,\theta_k\in[0,2\pi)$
are fixed. Then there exists $\alpha\in[0,2\pi)$
such that
$\tan(\alpha), \tan(\alpha-\pi/2), \tan(\alpha-\theta_1),\dots
\tan(\alpha-\theta_k)$ are all finite and badly
approximable.
\end{lemma}
\begin{proof}
This is Lemma 2.2 of \cite{Chen2007662}.
\end{proof}

Given an $\alpha$ as described by Lemma~\ref{lem:chen2.2},
they construct a list  of $N$ points in $[0,1]^2$
satisfying~\eqref{eq:lowdiscpts}.
To obtain their points they take the lattice
$N^{-1/2}\ints^2$ and rotate it through an angle $\alpha$
anticlockwise about the origin, and retain only
those points which lie in $[0,1]^2$. The result
will not necessarily have $N$ points, but by
adding or removing $O(\log(N))$ points 
they arrive at a set $\cp$ of $N$ points in $[0,1]^2$.

To apply their method, we will place points inside
the right angle triangle
\begin{align}\label{eq:righttri}
R = \Delta((0,0)^\tran,(0,1)^\tran ,(1,0)^\tran).
\end{align}
The sides of the
parallelograms of the form $\ct_{a,b,C}$, $\ct_{a,c,B}$
and $\ct_{b,c,A}$ for triangle $R$ make angles
$0$, $\pi/2$ and $3\pi/4$ (and no others) with respect to the horizontal axis.
Intersecting any of those parallelograms with $R$ always yields
a convex polygon whose sides make an angle of $0$, $\pi/2$ or $3\pi/4$ with
the horizontal axis. Lemma~\ref{lem:qirok} supplies for this set
of angles some choices for the  $\alpha$
whose existence is asserted by Lemma~\ref{lem:chen2.2}.

\begin{lemma}\label{lem:qirok}
Let $\alpha$ be an angle for which 
$\tan(\alpha)$ is a quadratic
irrational number. Then
$\tan(\alpha)$,  $\tan(\alpha-\pi/2)$
and $\tan(\alpha-3\pi/4)$ are all finite and badly approximable.
\end{lemma}
\begin{proof}
Write $\tan(\alpha) = (a+b\sqrt{c})/d$ 
for integers $a$, $b$, $c$ and $d$
satisfying $b\ne0$, $d\ne0$, and $c>0$,
with $c$ not the square of an integer.
First, $\tan(\alpha)$ is badly approximable
(and finite) because it is quadratic irrational.  Similarly,
$$\tan(\alpha-\pi/2) =  -\cot(\alpha) = 
\frac{d}{a+b\sqrt{c}} = \frac{da-bd\sqrt{c}}{a^2-b^2c}$$
is finite and badly approximable. Note that the denominator
in $\tan(\alpha-\pi/2)$  is not zero because
$c$ is not a perfect square, and $bd\ne 0$ too.
Finally,
$$\tan(\alpha-3\pi/4) = 
\frac{1+\tan(\alpha)}{1-\tan(\alpha)}
=
\frac{1+\frac{a+b\sqrt{c}}d}
{1-\frac{a+b\sqrt{c}}d}
=
\frac{(d-a)^2+b^2c + 2bd\sqrt{c}}{(d-a)^2-b^2c}
$$
is also finite and badly approximable. 
\end{proof}

As an example, $\tan(3\pi/8) = 1+\sqrt{2}$
is a quadratic irrational.
Therefore the angle $\alpha = 3\pi/8$ satisfies the
conditions of Lemma~\ref{lem:qirok}.

\begin{theorem}
Let $N>1$ be an integer and
let $R$ be the triangle given by \eqref{eq:righttri}.
Let $\alpha\in(0,2\pi)$ be an angle for which
$\tan(\alpha)$ is a quadratic irrational.
Let $\cp_1$ be the points of the lattice $(2N)^{-1/2}\ints^2$
rotated anticlockwise by angle $\alpha$.
Let $\cp_2$ be the points of $\cp_1$ that lie in $R$.
If $\cp_2$ has more than $N$ points, let $\cp_3$
be any $N$ points from $\cp_2$, or if $\cp_2$ has fewer
than $N$ points, let $\cp_3$ be a list of $N$ points
in $R$ including all those of $\cp_2$.
Then there is a constant $C$ with
$$D^\pr(\cp_3;R) < C\log(N)/N.$$
\end{theorem}
\begin{proof}
By Lemma~\ref{lem:qirok}, the angle $\alpha$ in the
hypothesis of this theorem satisfies the conditions
required by the construction of \cite{Chen2007662}.
We may therefore use that construction to get $2N$
points $\cp_1$ in $[0,1]^2$ such that their Theorem
1 yields $D_{2N}(\cp_1;\ca(\Theta)) < C_\Theta \log(2N)/(2N)$
where $\Theta = \{0,\pi/2,3\pi/4\}$.
Because the set $R\in \ca(\Theta)$ and has area
$1/2$, we know that the number of points in $\cp_2$
is between $N-C_\Theta \log(2N)/(2N)$
and $N+C_\Theta \log(2N)/(2N)$.
Then $\cp_3$  and $\cp_2$ differ by at most $C_\Theta\log(2N)/(2N)$
points, so that
$$D^\pr(\cp_3,R)<2C_\Theta \log(2N)/(2N)
=C_\Theta \log(2N)/N,$$
and we may take $C=2C_\Theta$.
\end{proof}

Because $D^\pr$ is invariant under linear
mappings, we may then map the points $\cp_3$
of $R$ linearly onto any triangle we desire
to sample, and attain the same discrepancy.
We note that \cite{Chen2007662} analyze their
procedure in a different way from how they
define it. To simplify notation, they scale the unit square up
to $U_N=[0,\sqrt{N}]^2$ and then rotate $U_N$
through an angle of $-\alpha$, and bound the discrepancy
of the corresponding scaled and rotated polygons. The
resulting discrepancy bounds apply either way.

Our lattice algorithm runs as follows.
Given a target sample size~$N$,
an angle $\alpha$ such as $3\pi/8$ satisfying Lemma~\ref{lem:qirok},
and a target triangle $\Delta(A,B,C)$,
\begin{compactenum}[\quad1)]
\item $n\gets\lceil\sqrt{2N}\rceil+1$
\item $\cp\gets \{-n,-n+1,\dots,n-1,n\}^2$
\item For each $\bsx_i$ in $\cp$, 
$\bsx_i\gets \bigl(\begin{smallmatrix}\cos(\alpha)&-\sin(\alpha)
\\ \sin(\alpha) & \phantom{-}\cos(\alpha)\end{smallmatrix}\bigr)
{\bsx_i}/{\sqrt{2N}}$
\item Remove points from $\cp$ that are not in $R$
\item (Optional) add points to or remove points from $\cp$ to get $N$ points in $R$
\item Linearly map $\cp$ from $R$ to $\Delta(A,B,C)$:
$A + (C-A)x_{i1} + (B-A)x_{i2}$.
\end{compactenum}

Steps $1$ and $2$ generate a subset of $\ints^2$
containing all the points that might possibly
end up in $R$ after rotation. Step $3$ does the rotation.
Step 4 retains those rotated points that lie in $R$.
Step 5 is optional;
in applications it may not be important to get precisely
$N$ points in $\Delta(A,B,C)$. For $\bsx_i = (x_{i1},x_{i2})^\tran$,
Step 6 maps $(0,0)^\tran$ onto $A$, $(0,1)^\tran$ onto $B$
and $(1,0)^\tran$ onto $C$.

Figure~\ref{fig:rotgrid} shows some points contructed
this way. Two of the examples use angles with badly approximable
tangents and the other two do not.  Those latter ones
leave some relatively large trapezoids nearly empty.

\begin{figure}
\center
\includegraphics[width=\hsize]{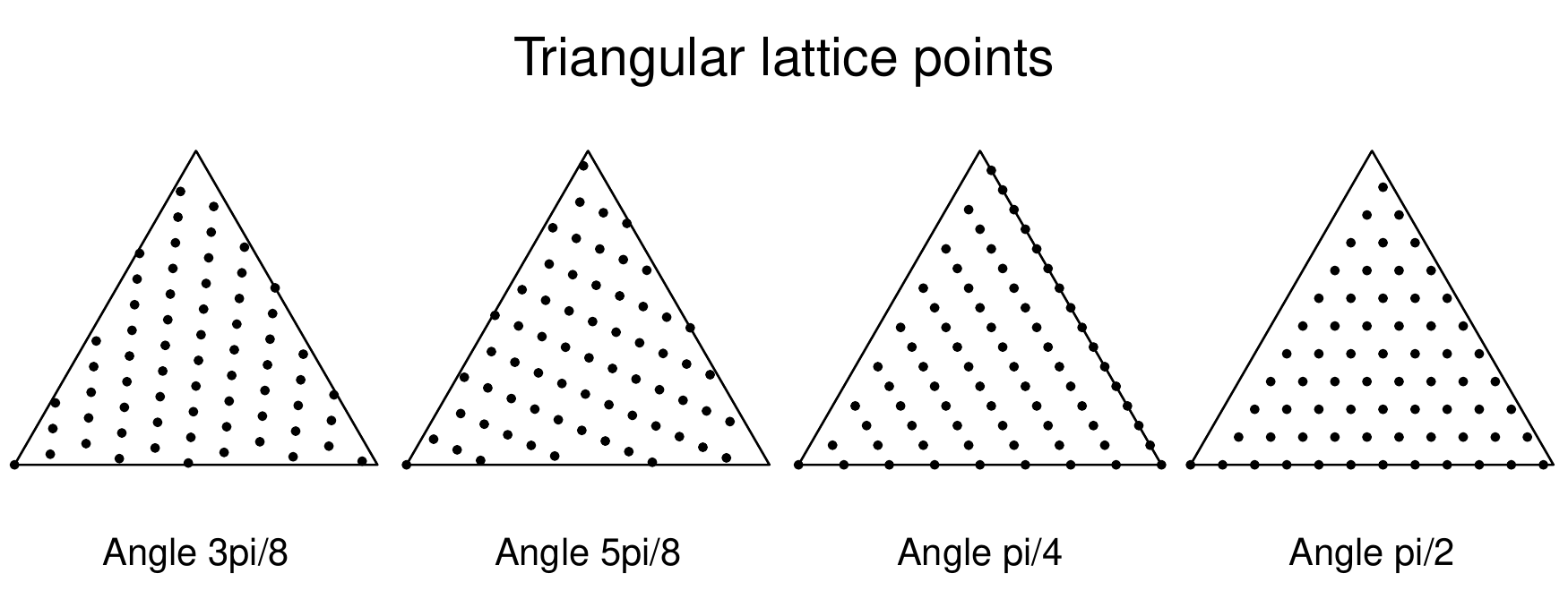}
\caption{\label{fig:rotgrid}
Triangular lattice points for target $N=64$.
Domain is an equilateral triangle.
Angles $3\pi/8$ and $5\pi/8$ have badly approximable
tangents.  Angles $\pi/4$ and $\pi/2$ have integer
and infinite tangents respectively and do not satisfy
the conditions for discrepancy $O(\log(N)/N)$.
}
\end{figure}

Figure~\ref{fig:discfig3piby8} plots the 
parallelogram discrepancy for angle $\alpha = 3\pi/8$.
We see that already for $N$ in the range $10$ to $500$
the discrepancy runs roughly parallel to the asymptotic
bound $O(\log(N)/N)$. Results from \cite{beck:chen:1987}
(cited in \cite{Chen2007662})
show that $D^\pr_N(\cp)$ cannot be $o(\log(N)/N)$.

\begin{figure}[t]
\center
\includegraphics[width=.8\hsize]{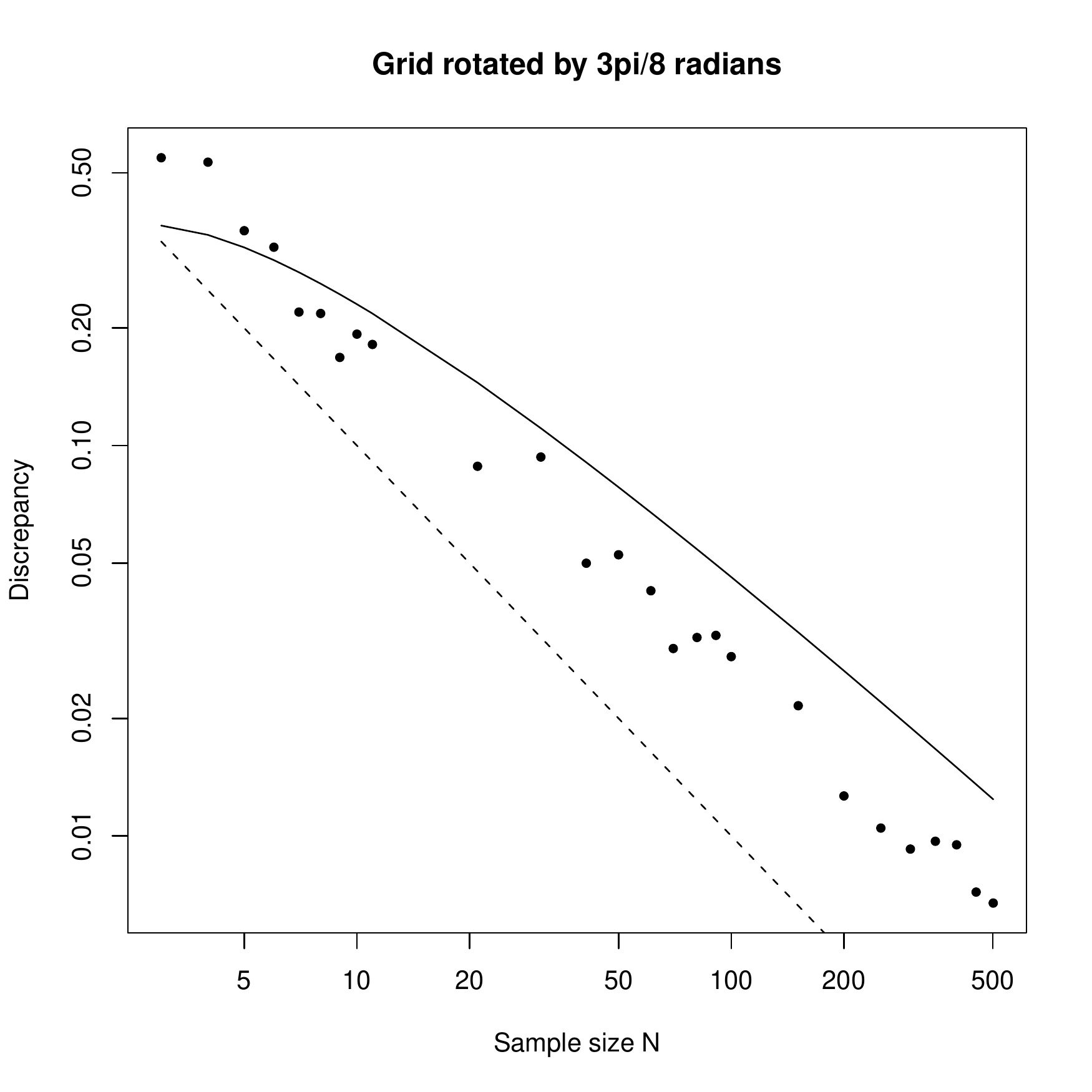}
\caption{\label{fig:discfig3piby8}
Parallel discrepancy of triangular lattice points
for angle $\alpha = 3\pi/8$ and various targets $N$.
The number of points was always $N$ or $N+1$.
The dashed reference line is $1/N$. The solid line
is $\log(N)/N$.
}
\end{figure}

We may want to randomly shift the points.
This can be done by adding a vector
$\bsU\sim\dustd(-1/2,1/2)^2$ to each point in
$\cp$ at step 2.
There are two benefits to randomly shifting
the points. First, with probability one there
will be no point rotated exactly on the boundary of
$R$, and we can then use simple averaging instead
of dividing the weighted sum~\eqref{eq:weightedf}
by $N$. The second advantage is that
can use independent repetitions of this randomization
to estimate error.

\section{Riemann integrable functions}\label{sec:riem}
The usual definition of Riemann Integral of a bounded function on a set in $\mathbb{R}^2$ can be found many books on, such as \cite{ash:dole:2000}
or \cite{mars:1974}. Here we develop an analogue for the triangle.

Let $T$ be nondegenerate closed triangle in $\real^2$.
For $k\ge0$ and $N=4^k$, let $T_{k,1},\dots,T_{k,N}$ be the partition of $T$
into $N$ congruent triangles, similar to $T$.
Let $f$ be a bounded function on $T$.
We say that $f$ is Riemann integrable on $T$ if
$$
\lim_{k\to\infty} \frac1{4^k}\sum_{i=1}^{4^k} f(\bsx_{k,i}) = \mu\in\real
$$
exists for any choices of $\bsx_{k,i}\in T_{k,i}$.
Then we take $\int_Tf(\bsx)\rd\bsx = \mu\times\vol(T)$.

For $k\ge1$ and $i=1,\dots,4^k$, let
\[\begin{split}
m_{k,i} = \inf\{f(\bsx) \mid \bsx \in T_{k,i}\}\quad\text{and}\quad
M_{k,i} = \sup\{f(\bsx) \mid \bsx \in T_{k,i}\}. 
\end{split}
\]
Then $f$ is Riemann integrable if and only if
$
\lim_{k\to\infty} \sum_{i=1}^{4^k} (M_{k,i}-m_{k,i})/4^k=0.
$
By modifying an argument in \cite[Theorem 1.6.6]{ash:dole:2000},
$f$ is Riemann integrable if it is bounded
and continuous almost everywhere on $T$.
When the Riemann integral exists, it matches the Lebesgue
integral.

\begin{figure}
\center
\includegraphics[width=\hsize]{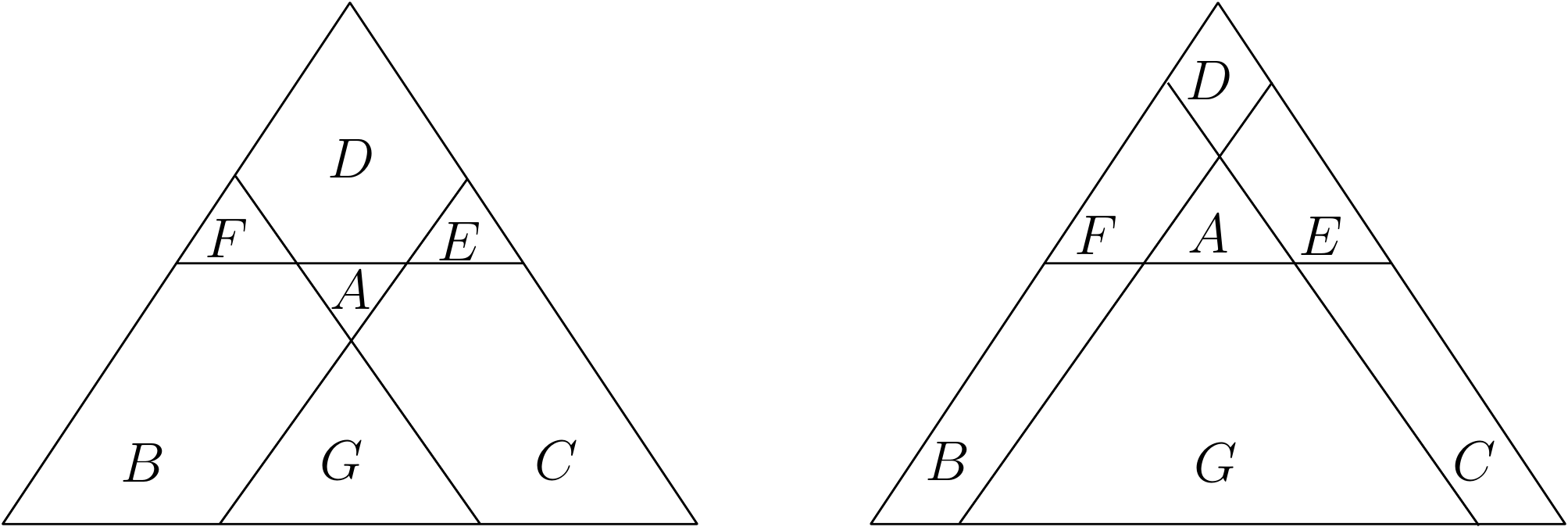}
\caption{\label{fig:tridisc}
Figure to illustrate decomposition of signed
discrepancies of upright (respectively inverted)
subtriangle $A$ in terms of parallelograms.
}
\end{figure}

\begin{lemma}\label{lem:tridisc}
Let $T$ be a triangle and $\cp$ a list of $N\ge1$ points
in $T$, having parallel discrepancy $D^\pr_N(\cp;T)$. 
Let $S$ be a subtriangle of $T$ with sides parallel
to those of $T$. Then $D_N(S;\cp) \le 6D^\pr_N(\cp;T)$.
\end{lemma}
\begin{proof}
First suppose that $S$ is inverted with respect to $T$
shown as $A$ in the left panel of Figure~\ref{fig:tridisc}.
Let the subsets $A$, $B$, $\dots$, $F$ indicated there
be disjoint, and have union $T$.
Let $AB$, $ABDF$ et cetera be unions of those sets.
Then with the signed discrepancy function 
$\delta(\cdot)=\delta_N(\cdot;\cp,T)$,
and these six sets in $\cs_\pr$,
\begin{align*}
&\phantom{=}\ \delta(AB)+\delta(AC)+\delta(AD)
-\delta(ABDF)-\delta(ACDE)-\delta(ACGB)\\
&=-\delta(B)-\delta(C)-\delta(D)-\delta(E)-\delta(F)-\delta(G)\\
&=\delta(A),
\end{align*}
using additivity of signed discrepancy, and
$\delta(ABCDEFG)=\delta(T)=0$. 
Next let $S$ be the upright triangle shown as $A$ in the
right panel of Figure~\ref{fig:tridisc}.
Then using $6$ sets from $\cs_\pr$,
\begin{align*}
&\phantom{=}\
 \delta(ABFG)+\delta(ADEF)+\delta(ACEG)
+\delta(B)+\delta(C)+\delta(D)\\
&=3\delta(A)+2\delta(B)+2\delta(C)+2\delta(D)+2\delta(E)+2\delta(F)+2\delta(G)\\
&=\delta(A).
\end{align*}
In either case, $|\delta(A)|\le 6D^\pr_N(\cp;T)$, and so $D_N(S;\cp)\le6D_N^\pr(\cp;T)$.
\end{proof}

\begin{theorem}
Let $f$ be a Riemann integrable function on 
a nondegenerate triangle $\Omega$, and 
let $\cp_N = (\bsx_{N,1},\bsx_{N,2},\dots,\bsx_{N,N})$
for $\bsx_{N,i}\in\Omega$.
If $\lim_{N\to\infty} D^\pr_N(\cp;\Omega)=0$, then
\[
\lim_{N\to\infty}
\frac{\vol(\Omega)}{N} \sum_{i=1}^{N} f(\bsx_{N,i}) = \int_{\Omega}f(\bsx)\rd\bsx.
\]
\end{theorem} 
\begin{proof}
Fix $\epsilon>0$ and then choose $k\ge0$ so that
$4^{-k}\sum_{i=1}^{4^k}(M_{k,i}-m_{k,i})<\epsilon$.
Let $T_1,\dots,T_{4^k}$ be the $4^k$ congruent 
subtriangles of $\Omega$.
Then
\begin{align*}
\frac{\vol(\Omega)}N
\sum_{i=1}^Nf(\bsx_{N,i})
&=\sum_{j=1}^{4^k}
\frac{\vol(\Omega)}N
\sum_{i=1}^N1_{\bsx_{N,i}\in T_j}f(\bsx_{N,i})\\
&\le\sum_{j=1}^{4^k}\frac{\vol(\Omega)}N
\sum_{i=1}^N1_{\bsx_{N,i}\in T_j}M_{k,j}\\
&=\vol(\Omega)\sum_{j=1}^{4^k}
\bigl(4^{-k} + \delta_N(T_j;\cp_N)\bigr)M_{k,j}.
\end{align*}
From Lemma~\ref{lem:tridisc}, $|\delta_N(T_j;\cp_N)|\le 6D_N^\pr(\cp_N)$.
Therefore
\begin{align*}
\frac{\vol(\Omega)}N
\sum_{i=1}^Nf(\bsx_{N,i}) \le \int_\Omega f(\bsx)\rd\bsx
+\epsilon\vol(\Omega) + 6\times4^kD_N^\pr(\cp_N)\vol(\Omega)|M_{0,1}|.
\end{align*}
Similarly,
$$
\frac{\vol(\Omega)}N
\sum_{i=1}^Nf(\bsx_{N,i}) \ge \int_\Omega f(\bsx)\rd\bsx
-\epsilon\vol(\Omega) - 6\times4^kD_N^\pr(\cp_N)\vol(\Omega)|m_{0,1}|.
$$
To complete the proof, let $N\to\infty$ and then note
that $\epsilon$ was arbitrary.
\end{proof}

\section{Discussion}\label{sec:disc}
The Kronecker construction attains a lower discrepancy
than the van der Corput construction.
But the van der Corput construction is extensible and
the digits in it can be randomized. If the integrand $f$
is continuously differentiable, then for $N=4^k$, the randomization
in \cite{rtms} will produce integral estimates with a root
mean squared error $O(N^{-1})$, slightly better than
the Koksma-Hlawka bound for the deterministic Kronecker construction.
As a result, we anticipate that both constructions will
be useful in applications. The variation measure
used by \cite{bran:colz:giga:trav:2013} requires even more
smoothness than one derivative and so 

Spherical triangles are also of interest.
The digital construction can be used
to generate points inside
a proper spherical triangle (all angles less than $\pi$)
with corners $A$, $B$ and $C$
if averages like $(A+B)/2$ are projected back to the
surface of the sphere, to get the midpoint of the arc
from $A$ to $B$.

\section*{Acknowledgments}
This work was supported by the U.S.\ National Science
Foundation under grant DMS-0906056.

\bibliographystyle{apalike}
\bibliography{qmc}
\end{document}